\documentclass[final, 11pt]{amsart}

\usepackage{bbm}
\usepackage{amssymb}
\usepackage{fancyhdr}
\usepackage{amscd}
\usepackage{amsthm}
\usepackage{graphicx}
\usepackage{amsmath}
\usepackage{verbatim}

\usepackage[english]{babel}
\usepackage[T1]{fontenc}
\usepackage[utf8x]{inputenc}
\usepackage{lmodern}
\usepackage{url}
\usepackage[usenames]{color}
\usepackage{graphicx}

\usepackage{showkeys}

\usepackage{color}

\theoremstyle{theorem}
\newtheorem{thm}{Theorem}[section]
\newtheorem{lem}[thm]{Lemma}
\newtheorem{cor}[thm]{Corollary}
\newtheorem{prop}[thm]{Proposition}

\theoremstyle{remark}

\theoremstyle{definition}
\newtheorem{ex}[thm]{Example}
\newtheorem{df}[thm]{Definition}

\numberwithin{equation}{section}

\newcommand{\ind}{\mathbbm{1}}
\newcommand{\EE}{\mathbb{E}}
\newcommand{\PP}{\mathbb{P}}

\newcommand{\RR}{\mathbb{R}}

\newcommand{\ud}{\mathrm{d}}

\newcommand{\8}{\infty}

\def\P{\mathbb{P}}
\def\R{\mathbb{R}}
\def\bF{\bar F}

\renewcommand{\a}{\alpha}
\newcommand{\eps}{\varepsilon}

\usepackage[top=3cm, bottom=3cm, outer=3.5cm, inner=3cm]{geometry}
\usepackage[unicode]{hyperref}
\usepackage{fancyhdr}

\hypersetup{
    bookmarks=true,         
    unicode=false,          
    pdftoolbar=true,        
    pdfmenubar=true,        
    pdffitwindow=false,     
    pdfstartview={FitH},    
    pdfcreator={Creator},   
    pdfnewwindow=true,      
    colorlinks=true,       
    linkcolor=blue,          
    citecolor=green,        
    filecolor=magenta,      
    urlcolor=black           
}

\begin{document}

\title{ \ Iterated random functions and regularly varying tails}
	\author[E. Damek, P. Dyszewski]{Ewa Damek, Piotr Dyszewski}

	\address{Instytut Matematyczny, Uniwersytet Wroclawski, Plac Grunwaldzki 2/4, 50-384 Wroclaw, Poland}
	\email{Ewa.Damek@math.uni.wroc.pl, pdysz@math.uni.wroc.pl}
	\thanks{The first author was partially supported by the NCN Grant UMO-2014/15/B/ST1/00060. The second author was partially supported by the National Science Centre, Poland (Sonata Bis, grant number DEC-2014/14/E/ST1/00588)}




\maketitle

\begin{abstract}
	We consider solutions to so-called stochastic fixed point equation $R \stackrel{d}{=} \Psi(R)$, where $\Psi $ is a random Lipschitz function and $R$ is a random variable independent of $\Psi$. Under the assumption that $\Psi$ can be approximated by the function $x \mapsto Ax+B$ we show that the tail of $R$ is comparable with the one of $A$, provided
	that the distribution of $\log  (A\vee 1) $  is tail  equivalent.  
	In particular we obtain new results for the random difference equation.
\end{abstract}

\section{Introduction}

Let $\{ \Psi_{n} \}_{n \geq 1}$ be a sequence of independent identically distributed (iid) random Lipschitz functions. We consider the Markov chain defined by 
\begin{equation*}
	R_{n+1} = \Psi_{n+1}(R_n), \quad n\geq 0, 
\end{equation*}
where $R_0 $ is a random variable independent of $\{ \Psi_{n} \}_{n \geq 1}$.  Under rather mild moment assumptions, the Markov chain $\{R_n\}_{n \geq 0}$ possesses a unique stationary distribution.
Suppose that $R$ is distributed according to it and let $\Psi$ be a generic copy of  $\Psi_n $ independent of $R$, then necessarily
\begin{equation*}
	R \stackrel{d}{=} \Psi(R), \quad R \mbox{ independent of }\Psi,
\end{equation*}
where $\stackrel{d}{=}$ denotes the equality in distribution. Distributional equations of this form appear in wide range of problems in applied probability. Beginning from the early nineties  iterated function systems of i.i.d. Lipschitz maps (IFS) on a complete metric space have attracted a lot of attention: Alsmeyer \cite{alsmeyer2016}, Arnold and Crauel \cite{arnoldcrauel1992}, Brofferio and Buraczewski \cite{broferio2015}, Buraczewski and Damek \cite{buraczdamek2016},
Diaconis and Friedman \cite{diaconisfreedman1999}, Duflo \cite{duflo1997}, Elton \cite{elton1990multiplicative}, Henion and Herv\'e \cite{hennionherve2004}, Mirek \cite{mirek2011} and they still do. In particular, it seems that modelling them after random difference equation (described below) has been very fruitful, see Alsmeyer \cite{alsmeyer2016} and Mirek \cite{mirek2011}.

 \medskip

The main example we wish to present here, is the so-called random difference equation occurring whenever $\Psi$ is just 
affine transformation, i. e. $\Psi(x) = Ax+B$. Then, the recursive formula for the Markov chain in question which, in this special instance, will be denoted by $\{X_n \}_{n \geq 0}$, takes the simple form
\begin{equation*}
	X_{n+1} =A_{n+1}X_n +B_{n+1}, \quad n \geq 0,
\end{equation*}
where $\{ (A_n,B_{n}) \}_{n \geq 1 }$ is a sequence of iid two-dimensional random vectors. Here the stochastic equation satisfied 
by $X$ is  
\begin{equation*}
	X \stackrel{d}{=} AX+B, \quad X \mbox{ independent of }(A,B),
\end{equation*}
where $(A,B)$ is an independent copy of $(A_n, B_n)$. It turns out, that due to the explicit expression of the function $\Psi(x) = Ax+B$, the stationary solution can be explicitly represented  by
\begin{equation*}
	X \stackrel{d}{=} \sum_{k=0}^{\infty} B_{k+1}\prod_{j=1}^kA_j
\end{equation*}
provided that the series is convergent. The series above can be interpreted as the current value of future payments 
represented by $B_{k+1}$ with discount factors represented by $A_j$ and therefore, it  
is very often called a perpetuity. Random variables of this form appear also in context of Additive Increase Multiplicative 
Decrease algorithms~\cite{guillemin2004aimd} or  COGARCH processes~\cite{kluppelberg2006continuous}, to name a few. For more detailed discussion on perpetuities and related processes we refer the reader to recent monographs~\cite{buraczewski2016stochastic, iksanov2016renewal}. \medskip

From the point of view of applications the key information about the distribution of $X$ is its tail asymptotic, that is 
$$
	\PP[X>t] \quad \mbox{as} \quad t \to \infty.
$$ 
We wish to recall several scenarios, were $X$ exhibits regularly varying tail. Assume that $A, B \geq 0$ for the moment. The first example  is related to the work of Kesten~\cite{kesten1973random} 
and Goldie~\cite{goldie1991implicit}, which shows that if 
\begin{equation}\label{eq:kesten}
	\EE[A^{\alpha}]=1, \quad \EE[B^{\alpha}]<\infty, 
\end{equation}
for a positive $\alpha$, the stationary distribution has power tail, i. e.
\begin{equation}\label{KGtail}
	\PP[X>t] \sim c_X t^{-\alpha}\quad \mbox{as} \quad t \to \infty,
\end{equation}
for some implicitly given constant $c_X>0$. Here, and in what follows, for two functions $g$, $h$,  by $g(t) \sim h(t)$ we mean $\lim_{t \to \infty}\frac{g(t)}{h(t)} =1$. The asymptotic  \eqref{KGtail} follows from the behaviour of $A$, more precisely it is a direct consequence of the first assumption in~\eqref{eq:kesten}. It may as well happen that a heavy tail of $X$ is caused  by a heavy tail of $B$.
More precisely, the work of Grincevi{\'c}ius~\cite{grincevicius1975one} which was later improved by Grey~\cite{grey1994regular} treats the case 
\begin{equation}\label{eq:grey}
	\PP[B>t] \sim t^{-\alpha}L(t), \quad \EE[A^{\alpha}]<1,
\end{equation}
where $\alpha>0$ and $L(t)$ is a slowly varying function, that is $L(ct) \sim L(t)$ for any fixed $c>0$. Then
$$
	\PP[X>t] \sim \frac{1}{1-\EE[A^{\alpha}]} \PP[B>t].
$$
Finally, in the recent work Kevei~\cite{kevei2016note} proves that, if
\begin{equation}\label{eq:kevei}
	\PP[A>t]\sim t^{-\alpha}L(t), \quad \EE[A^{\alpha}]<1, \quad \EE[B^{\alpha+\varepsilon}] <\infty, \footnote{ Plus some additional technical assumptions}
\end{equation}
for $\alpha, \varepsilon>0$, then
$$
	\PP[X>t] \sim \frac{\EE[X^{\alpha}]}{1-\EE[A^{\alpha}]} \PP[A>t].
$$
One important feature the scenarios~\eqref{eq:kesten}-\eqref{eq:kevei} have in common is that either $A$ or $B$ contributes significantly to the asymptotic of $X$, not both.\medskip

In order to get a more detailed information about the structure of the distribution of $X$ it is natural to consider the frontiers of the scenarios in question, where both coefficients have influence on the tail asymptotic of $X$. First one, situated between~\eqref{eq:kesten} and~\eqref{eq:grey} was recently investigated by Damek and Ko\l{}odziejek~\cite{damek2017stochastic}, is the case when
$$
	\EE[A^{\alpha}]=1, \quad \PP[B>t]\sim t^{-\alpha}L(t), \quad \EE[B^{\alpha}] = \infty
$$
which results in
$$
	\PP[X>t] \sim t^{-\alpha} \widetilde{L}(t),
$$
with some explicitly given slowly varying function $\widetilde{L}$. 

The second one, being the frontier between~\eqref{eq:grey}-\eqref{eq:kevei}, is
\begin{equation}\label{eq:new}
	\EE[A^{\alpha}]<1, \quad \PP[A>t] \asymp \PP[B>t] \asymp t^{-\alpha}L(t),
\end{equation}
where $g(t) \asymp h(t)$ means that
\begin{equation*}
	0< \liminf_{t \to \infty}\frac{g(t)}{h(t)} \leq \limsup_{t \to \infty}\frac{g(t)}{h(t)} <\infty. 
\end{equation*}
Up to our best knowledge, this case was not studied in the literature apart form two specific cases: independent $A$ and $B$ treated in~\cite{li2015interplay} and 
so-called exponential functional of L{\'e}vy processes studied in~\cite{rivero2012tail}. Our aim is to present a robust approach to treat the scenario~\eqref{eq:new} and its counterpart for the iterated random functions $\{R_n\}_{n \geq 0}$ and $R$. \medskip 

We will work under the assumption that $\Psi$ can be well approximated by the affine transformation, that is
\begin{equation*}
	\left|\Psi(x) - Ax\right| \leq B \phi(x),
\end{equation*}
with 
\begin{equation*}
	\PP[B>t] = O(\PP[A>t]), \quad \phi \ \mbox{ is locally bounded } \quad  \mbox{and} \quad  \phi(x) = o(x), \ \mbox{as}\ x\to \8
\end{equation*}
and, among some technical assumptions, that for $\alpha>0$
$$
	\PP[A>t] \sim t^{-\alpha}L(t).
$$
In order to be able to successfully treat the case $\PP[A>t] \asymp \PP[B>t]$ we will need to ensure that the successive iterations of $\{\Psi_n \}_{n \geq 1}$ are well-behaved, i.e. for fixed $x$, 
$$
	\PP[\Psi(x)>t]\sim f_+(x) \PP[A>t]
$$
for some measurable function $f_+$. Under the above, the main result of this article states that
\begin{equation*}
	\PP[R>t] \sim \frac{\EE[f_+(R)]}{1-\EE[A^{\alpha}]} \PP[A>t],\quad \mbox{as}\ t \to \infty ,
\end{equation*}
with $\EE[f_+(R)]<\infty$. \medskip 

The article is organized as follows. In Section~\ref{sec:eqtails} we recall some basic notions related to the class of convolution equivalent distributions. 
The results in the case of Random Difference Equation are stated in Section \ref{sec:rde} and in the case of Iterated Random Functions in Section~\ref{sec:ifs}. The proofs are presented in Section~\ref{sec:proofs}. The Appendix contains proofs of some 
classical properties of the convolution equivalent distributions.

\section{Convolution equivalent tails}\label{sec:eqtails}
Throughout the paper we would like to benefit from properties of convolution equivalent distributions. We begin by introducing some basic notation. We will consider distribution, say $F$, with right-unbounded support. Write $F^{*n}$ for 
$n$th-convolution of $F$ and $\overline{F}$ for its tail, that is $\overline{F}(x) = 1-F(x)$.  

\begin{df}\label{df:sa}
	A distribution $F$ with right-unbounded support contained in $ [0, + \infty)$ is said to be tail equivalent if for any fixed $y \in \RR$, as $t \to \infty$
	\begin{equation}\label{eq:Lalpha}
		\overline{F}(t-y) \sim \overline{F}(t)e^{\alpha y},
	\end{equation}
	$m_{\alpha}(F)=\int e^{\alpha s} F(ds)<\infty$ and moreover
	\begin{equation}\label{eq:conveq}
		\overline{F^{*2}}(t) \sim 2m_{\alpha}(F) \overline{F}(t)
	\end{equation}
	for some $\alpha \geq 0$. In that case we will write $F \in \mathcal{S}(\alpha)$. By slight abuse of notation we will write 
	for random variable $X$, $X \in \mathcal{S}(\alpha)$ whenever its distribution is a member of $\mathcal{S}(\alpha)$.
\end{df}
This class was introduced independently by Chistyakov~\cite{chistyakov1964theorem}  and 
Chover et al.~\cite{chover1973functions, chover1973degeneracy}. 
The key feature of the class of convolution equivalent distributions is that only the right tail behaviour is of significance. For this reason it is natural to work with a wider class of distributions supported on the whole real line $\RR$.
\begin{df}\label{df:saR}
	We will say that a distribution $F$ with right-unbounded support contained in $\RR $ is tail equivalent if $F$ satisfies~\eqref{eq:Lalpha} and~\eqref{eq:conveq}. If this is the case, we will write $F \in \mathcal{S}_{\RR}(\alpha)$.
\end{df} 

It is not difficult to see that $X\in  \mathcal{S}_{\RR}(\alpha)$ if, and only if $X$ conditioned on the set $\{ X\geq 0\}$ is in $\mathcal{S}(\alpha )$. Equivalently, in term of the distributions 
$$
	F \in \mathcal{S}_{\RR}(\alpha) \quad \Leftrightarrow \quad F^+(t) = \frac{F(t) - \overline{F}(0)}{\overline{F}(0)} \ind_{[0,\infty)}(t) \in \mathcal{S}(\RR).
$$
Next property of distributions form the class $\mathcal{S}_{\RR}(\alpha)$ will be particularly important for us. 

\begin{lem}\label{lem:convsa}
	Assume that $F \in \mathcal{S}_{\R}(\alpha)$. If for some distributions $G_1$, $G_2$, $\overline{G}_i(t) \sim k_i\overline{F}(t)$,
	then
	\begin{equation}\label{convol}
		\lim_{t \to \infty} \frac{\overline{G_1*G_2}(t)}{\overline{F}(t)} = k_1 m_{\alpha}(G_2) + k_2 m_{\alpha}(G_1),
	\end{equation} 
	where $m_{\alpha}(G_i)=\int e^{\alpha s} G_i(ds)$. Moreover, $k_i>0$ implies $G_i \in \mathcal{S}_{\R}(\alpha)$. 
	If on the other hand
	$$
		\limsup_{t \to \infty}\frac{\overline{G}_i(t) }{\overline{F}(t)} \leq k_i,
	$$
	then 
	\begin{equation}\label{convol1}
	\limsup_{t \to \infty} \frac{\overline{G_1*G_2}(t)}{\overline{F}(t)}\leq  k_1 m_{\alpha}(G_2) + k_2 m_{\alpha}(G_1)
	\end{equation}
	and there is a function $\eta (t)\geq 0$, $\lim _{t\to \8}\eta (t)=0$ such that for every $t$
	\begin{equation}\label{convol2}
		\overline{G_1*G_2}(t)\leq  (k_1 m_{\alpha}(G_2) + k_2 m_{\alpha}(G_1)+(k_1k_2+m_{\alpha}(G_1)+m_{\alpha}(G_2))\eta (t)) \overline{F}(t).
	\end{equation}
	Similarly, if 
	$$
		\liminf_{t \to \infty}\frac{\overline{G}_i(t) }{\overline{F}(t)} \geq  k_i,
	$$
	then 
	\begin{equation}\label{convol3}
	\liminf_{t \to \infty} \frac{\overline{G_1*G_2}(t)}{\overline{F}(t)}\geq  k_1 m_{\alpha}(G_2) + k_2 m_{\alpha}(G_1)
	\end{equation}
	and for every $t$
	\begin{equation}\label{convol4}
		\overline{G_1*G_2}(t)\geq  (k_1 m_{\alpha}(G_2) + k_2 m_{\alpha}(G_1)-(k_1k_2+m_{\alpha}(G_1)+m_{\alpha}(G_2))\eta (t)) \overline{F}(t).
	\end{equation}

\end{lem}
For completeness reasons, the proofs of the above Lemma, and some other discussions regarding the class $\mathcal{S}_{\R}(\alpha)$,  can be found in the Appendix. 
In the case of $\mathcal{S}(\alpha)$, for \eqref{convol}, one would classically refer to~\cite{Cline1986}.  \medskip

Condition~\eqref{eq:conveq} present in Definition~\ref{df:sa} seems to be and in fact it is technical. Since $\mathcal{S}_{\R}(\alpha)$ is the class we are mainly interested in, before we proceed any further we will present some sufficient conditions for 
$F \in \mathcal{S}_{\R}(\alpha)$. As it was proved by Kl{\"u}ppelberg~\cite{kluppelberg1989subexponential} (see Theorem~2.1), for 
$\alpha >0$
\begin{equation}
	F \in \mathcal{S}(\alpha ) \quad \Leftrightarrow \quad h(x) = e^{\alpha x}\overline{F}(x) \ind_{[0,\infty)}(x) \in \mathcal{S}_d,
\end{equation}  
 where $\mathcal{S}_d$ denotes the class of subexponential densities, namely $h \in \mathcal{S}_d$ if
\begin{equation*}
	h(x-y) \sim h(x)
\end{equation*}
for any fixed $y \in \RR$ as $x \to \infty$, $m_0 = \int_0^{\infty} h(y) dy < \infty$ and
\begin{equation*}
	\int_0^x h(x-y)h(y) dy \sim 2m_0 h(x).
\end{equation*}
Knowing sufficient conditions for $ h \in \mathcal{S}_d$, here Theorems 4.15 and 4.16 in~\cite{foss2011introduction}, we can rewrite those in terms of $F$ and obtain the next two Corollaries. 

\begin{cor}\label{cor:dom}
	Assume that $\overline{F^+}(t) \sim e^{-\alpha t} K(t)$, where $\alpha>0$, $K(x-y) \sim K(x)$ for any fixed $y \in \RR$ as 
	$x \to \infty$. 
	If one can find a constant $c>0$ for which $K(2x) \geq cK(x)$ for sufficiently large $x$, then $F \in \mathcal{S}_{\R}(\alpha)$ provided that 
	$\int K(t) dt <\infty$.
\end{cor}

\begin{cor}\label{cor:convex}
	Suppose we have $\overline{F^+}(t)\sim e^{-\alpha t} K(t)$ for $\alpha>0$.
	If $ -\log K(x)$ is eventually concave and one can find a function $f\colon \RR \to \RR$ such that	
	\begin{itemize}
		\item $f(x) \leq x/2$ but $f(x) \to \infty$ as $x \to \infty$,
		\item $K$ if $f$-insensitive, i.e. $K(x-y) \sim K(x)$ as $x \to \infty$, uniformly in $y \leq f(x)$,
		\item $xK(f(x)) \to 0$ as $x \to \infty$,
	\end{itemize}
	then $F \in \mathcal{S}_{\R}(\alpha)$ if additionally $\int K(t) dt <\infty$. 
\end{cor}

\begin{ex}
	By Corollary~\ref{cor:dom}, if $\overline{F^+}(t) \sim c e^{-\alpha t} t^{p}$ for $t>t_0$, $\alpha, c >0$ and $p<-1$ then 
	$F \in \mathcal{S}_{\R}(\alpha)$. If on the other hand $\overline{F^+}(t) \sim c \exp\{-\alpha t -\beta t^{\gamma} \}$ for 
	$\alpha, \beta, c >0$ and $\gamma \in (0,1)$ then again $F \in \mathcal{S}_{\R}(\alpha)$ but this time by Corollary~\ref{cor:convex}
	with $f(x) = \log^{1/\gamma}(x)$.
\end{ex}

Other sufficient conditions, going beyond Corollaries~\ref{cor:convex} and~\ref{cor:dom} can be found in~\cite{embrechts1983property} and \cite{Cline1986}.

\section{Random Difference Equation}\label{sec:rde}

We will start with the case when $\Psi(x) = Ax+B$, in order to introduce the set-up to the problem and deliver some enlightening examples. For the sake of transparency, throughout this section we will assume that $A>0$ a.s. The results in full generality, including the case of two-sided $A$ will be treated in Section~\ref{sec:ifs}. For the needs of this Section, one can just take an iid sequence of two-dimensional random vectors $\{ (A_n, B_n)\}_{n \geq 1}$, with $A_n>0$, and consider a~Markov chain given via
\begin{equation}\label{eq:defXn}
	X_{n} = A_{n}X_{n-1}+B_{n} \quad \mbox{for } n \geq 1.
\end{equation}
The only condition we impose on $X_0$ at this point is independence form $\{(A_n, B_n)\}_{n \geq 1}$. By a well-known fact, if 
\begin{equation*}
	\EE [\log(A)] < 0 \quad \mbox{and} \quad \EE [\log(|B|+1)] <\infty,
\end{equation*}
then the Markov chain $\{X_n\}_{n \geq 0}$ possesses a~unique stationary distribution which can be represented by a random variable of the form
\begin{equation}\label{eq:perpetuity}
	X \stackrel{d}{=} \sum_{k=0}^{\infty} B_{k+1}\prod_{j=1}^kA_j,
\end{equation}  
see~\cite{vervaat1979stochastic} for the above or \cite{goldie2000stability} of necessary and sufficient conditions for the convergence. By the stationary, $X$ will be  a solution to the stochastic equation 
\begin{equation}\label{eq:ax+b}
	X \stackrel{d}{=} AX+B, \quad \mbox{$X$ independent from $(A,B)$}.
\end{equation}
We would like to investigate $\PP[X>t]$ in the case, where $A$ and $B$ have comparable tails. We will work under the assumption that  $\log A \in \mathcal{S}_{\R}(\alpha)$. To state the conditions in Definition~\ref{df:saR} explicitly, we will consider $A$ with regularly varying tail, namely for any $ y >0$ satisfying 
\begin{equation}\label{eq:Atail1}
	\PP[A> ty] \sim y^{-\alpha}\PP[A>t]
\end{equation}
as $t \to \infty$ and $\EE[A^{\alpha}]<\infty$. Moreover, denoting by $A'$ an independent copy of $A$, assume that
\begin{equation}\label{eq:tailA2}
\PP[AA'>t] \sim 2 \EE[A^{\alpha}]\PP[A>t]
\end{equation}
for some $\alpha >0$. The case of $\alpha =0$, when $\mathcal{S}(0) = \mathcal{S}$ is the class of subexponential distributions, was treated in \cite{dyszewski2016iterated, korshunov2016look, tang2016random}. 
To ensure that Cram\'er's condition is not satisfied, assume
\begin{equation}\label{eq:Amoment1}
	\EE[A^{\alpha}] <1.
\end{equation}
At this point it is worth noting that condition \eqref{eq:Atail1}  implies in particular that for any $\varepsilon>0$
\begin{equation*}
	\EE[A^{\alpha +\varepsilon}] =\infty,
\end{equation*}
see for example~\cite{embrechts1982convolution}.
As a particular consequence, the results of Grey~\cite{grey1994regular} will also not apply directly. However, as we will see, one can use a similar approach as the one presented in~\cite{grey1994regular}.

Under the above, the tails of $X$ and $A$ are weakly equivalent, provided that the tail of $B$ is of the same order. Note, that if  
\begin{equation}\label{eq:tailsB}
	\limsup_{t \to \infty} \frac{\PP[|B|>t]}{\PP[A>t]} = c_{|B|} < \infty.
\end{equation}
then in particular $\EE[|B|^{\alpha}]< \infty$.
In view of~\eqref{eq:Amoment1},  Minkowski's inequality entails 
\begin{equation*}
		\EE[|X|^{\alpha}]<\infty,
\end{equation*}
for details we refer to Alsmeyer et al.~\cite{alsmeyer2009distributional} or to Section~\ref{sec:proofs}. Without any further assumptions, we were able to prove, that the tails of $X$ and $A$ are weakly equivalent. Next Proposition will follow form our main result, presented in the Section~\ref{sec:ifs}.

\begin{prop}\label{prop:weak}
	Suppose $ A> 0$ and that conditions \eqref{eq:Atail1} - \eqref{eq:tailsB} hold true. Then the Markov chain $\{ X_n \}_{n\geq 0}$ converges 
	weakly to $X$. Moreover, as $t \to \infty$,
	\begin{equation*}
		\PP[|X|>t] = O(\PP[A>t]).
	\end{equation*}
	Furthermore, if $\PP[A>t, B < - t] = o(\PP[A>t])$, then
	\begin{equation*}
		\PP[|X|>t] \asymp \PP[A>t].
	\end{equation*}
\end{prop}

At this point we are obliged to mention that the constants we obtain in the claims of Proposition~\ref{prop:weak} are not optimal. Since our main goal is establishing the precise asymptotic of  $\PP[X>t]$ we will not pursue the optimal constants in Proposition~\ref{prop:weak}. \medskip
 
To be able to determine the exact asymptotic of $\PP[X>t]$ some additional conditions need to be imposed.  Namely, assume that 
\begin{equation}\label{eq:tailequiv+}
	\lim_{t \to \infty}\frac{\PP[Ay+ B>t]}{\PP[A>t]}=  f_+(y), \quad \mbox{for } y \in {\rm supp}\mathcal{L}(X),
\end{equation}
and
\begin{equation}\label{eq:tailequiv-}
	\lim_{t \to \infty}\frac{\PP[Ay+ B<-t]}{\PP[A>t]}=  f_-(y), \quad \mbox{for }y \in {\rm supp}\mathcal{L}(X),
\end{equation}
where $f_\pm \colon \RR \to [0, +\infty)$ are some measurable function and $\mathcal{L}(X)$ denotes the distribution of $X$.  Imposing~\eqref{eq:tailequiv+} and \eqref{eq:tailequiv-} will allow us to investigate the case of dependent $A$ and $B$ with comparable tails. Note that under the above $\PP[|B|>t] \sim (f_+(0) + f_-(0)) \PP[A>t]$, so that~\eqref{eq:tailequiv+} and \eqref{eq:tailequiv-} imply~\eqref{eq:tailsB}.  As one of the consequences coming from combining conditions~\eqref{eq:tailequiv+} and~\eqref{eq:Atail1} is a bound for function $f_+$. Namely, we may write for any $y \geq 0$
\begin{align*}
	\PP[Ay + B >t]  & = \PP[Ay > t/2, \: Ay+ B>t] + \PP[Ay\leq t/2, \: Ay+B>t]\\
	&  \leq \PP[Ay > t/2] + \PP[B>t/2] \sim \left((2y)^{\alpha} + 2^{\alpha}f_+(0)\right) \PP[A>t],
\end{align*}
while for $y \leq 0$, 
\begin{equation*}
	\PP[Ay+B > t] \leq \PP[B>t] \sim f_+(0)\PP[A>t]. 
\end{equation*}
Whence, since $A > 0$, for $y \in \RR$,
\begin{equation*}
	f_+(y) \leq 2^{\alpha}(y_+^{\alpha}+f_+(0)),
\end{equation*}
where $x_+=x^+ = \max \{ 0, x\}$. Thus, for example $\EE[f_+(X)]<\infty$. In a similar fashion we obtain the bound for $f_-$ of the form
$$
	f_-(y) \leq 2^{\alpha}(y_-^{\alpha} + f_-(0))
$$
and as a consequence $\EE[f_-(X)]<\infty$. Denote 
$$
		\mu_{\pm} = \EE[A_{\pm}^{\alpha}].
$$ 
Assuming the presented conditions, we aim to prove the following result.  

\begin{thm}\label{thm:perpetuity}
	Assume \eqref{eq:Atail1} - \eqref{eq:tailequiv-} and that $A > 0$ a.s. The Markov chain $\{ X_n \}_{n\geq 0}$ 
	converges weakly to $X$. 
	Moreover,
	\begin{equation*}
		\PP[X>t] \sim \frac{\EE[f_+(X)]  }{1-\mu_+  } \PP[A>t].
	\end{equation*}
\end{thm}

Note that under the assumptions of Theorem~\ref{thm:perpetuity} the same comment may be made regarding the left tail of $X$. More precisely, 	
\begin{equation*}
	\PP[X<-t] \sim \frac{\EE[f_-(X)]  }{1-\mu_+  } \PP[A>t].
\end{equation*}

To see a few examples, how~\eqref{eq:tailequiv+} and \eqref{eq:tailequiv-} come into play, we state the following Corollary, which 
treats the case when the tail of $A\wedge B$ is negligible. This covers the possibility that $A$ and $B$ are independent, as treated in~\cite{li2015interplay}, and the possibility that the tail of $B$ is negligible, as treated by Kevei~\cite{kevei2016note}. For simplicity we will assume that $B\geq 0$ so that $f_-(y)=0$ for any $y \in  \mathcal{L}(X) \subseteq[0,+\infty)$.

\begin{cor}\label{cor:indep}
	Assume $A,B \geq 0$ and that \eqref{eq:Atail1} - \eqref{eq:Amoment1} hold and moreover that
	\begin{equation*}
		\PP[B>t]\sim c_B\PP[A>t], \quad \PP[A>t, \:B>t] = o(\PP[A>t]),
	\end{equation*}
	for some $c_B \geq 0$. Then, as $t \to \infty$,
	\begin{equation*}
		\PP[X>t] \sim \frac{\EE[X_+^{\alpha}] +c_B }{1 -\EE[A^{\alpha}]  } \PP[A>t].
	\end{equation*}
\end{cor}
\begin{proof}
	We will invoke Theorem~\ref{thm:perpetuity}. To see why \eqref{eq:tailequiv+} holds with $f_+(y)=y_+^{\alpha} +c_B$, take an arbitrary $\delta \in (0,1)$ and consider the following two 
	bounds. For the upper one write
	$$
		\PP[Ay+B>t] \leq \PP[Ay>(1-\delta)t] + \PP[B>(1-\delta)t] + \PP[Ay>\delta t, \: B> \delta t],
	$$
	where the last term on the right-hand side is negligible, since in can be bounded viz.
	$$
		\PP[Ay>\delta t, \: B> \delta t] \leq \PP[A (y \vee 1) >\delta t, \: B (y\vee 1)>\delta t] = o(\PP[A> (y\vee 1)^{-1}\delta t]) = o(\PP[A>t]).
	$$
	For the lower bound one can just write simply that
	$$
		\PP[Ay+B>t] \geq \PP[Ay>t] + \PP[B>t] - \PP[Ay>t, \: B>t],
	$$
	where again the last term is negligible. Taking first $t \to \infty$ and then $\delta \to 0$ yields the desired result.
\end{proof}
In turns out that in the case when $\PP[A>t] \sim \PP[B>t]$ the knowledge only of marginals of $A$ and $B$ is insufficient to determine $\PP[X>t]$. We note that by the example in the same vein as the one presented in~\cite{dyszewski2016iterated}.
\begin{ex}
	We wish to compare the tails of $X$ with two types of input, i.e. two different vectors $(A,B)$ with the same marginals. Take any 
	 positive random variable $Z$ such that $\log Z \in \mathcal{S}(2)$ and denote $\mu = \EE[Z]$, $\sigma = \EE[Z^2]<1$. Firstly, consider $A^{(1)}$ and $B^{(1)}$  independent, both distributed as $Z$. 
	Then, by Corollary~\ref{cor:indep} for
	\begin{equation}\label{eq:example}
		X^{(1)} \stackrel{d}{=} A^{(1)}X^{(1)}+B^{(1)}
	\end{equation}
	one has
	\begin{equation*}
		\PP[X^{(1)}> t]  \sim  d_1\PP[Z>t].
	\end{equation*}
	Since the first and the second moment of $X^{(1)}$ that can be computed explicitly using \eqref{eq:example}, we have 
	\begin{equation*}
		d_1 = \frac{2\mu^3 - \mu+1}{(1-\mu)(1-\sigma)^2} .
	\end{equation*}
	For the second input consider $A^{(2)} = B^{(2)}$ with the same distribution as $Z$. Then
	\begin{equation*}
		X^{(2)} \stackrel{d}{=} A^{(2)}X^{(2)}+B^{(2)}
	\end{equation*}
 	can be written as
	\begin{equation*}
		X^{(2)} +1\stackrel{d}{=} A^{(2)}(X^{(2)}+1)+1.
	\end{equation*}
	Invoking Corollary~\ref{cor:indep} once again yields
	\begin{equation*}
		\PP[X^{(2)}> t] \sim \PP[X^{(2)} +1> t]  \sim  d_2\PP[Z>t],
	\end{equation*}
	where
	\begin{equation*}
		d_2 = \frac{2\mu +\sigma (1-\mu)}{(1-\mu)(1-\sigma)^2} .
	\end{equation*}
	Summarizing, $A^{(1)} \stackrel{d}{=} A^{(2)}$, $B^{(1)} \stackrel{d}{=} B^{(2)}$ but the asymptotic of tails of $X^{(1)}$
 	and $X^{(2)}$ are in general different, since $d_1$ differs form $d_2$.
\end{ex}
The Example above shows, that in order to determine the exact asymptotic of $X$ in the case when the tails of $A$ and $B$ are comparable, we need to have some information regarding the joint distribution of the vector $(A,B)$. One example of such information is encrypted in conditions~\eqref{eq:tailequiv+} and \eqref{eq:tailequiv-}.

\section{Iterated random functions}\label{sec:ifs}
Natural direction, in which one can generalize Theorem~\ref{thm:perpetuity} is by allowing $A$ to take negative values. Another one consists of replacing the function $x \mapsto Ax+B$ by a random, Lipschitz function $\Psi$. We aim to obtain both these generalizations in this Section, where we will give a statement of our main result in full generality.
We will now consider a Markov chain with more general form than~\eqref{eq:defXn}, namely
\begin{equation}\label{eq:defIFS}
	R_{n} = \Psi_{n}(R_{n-1}),
\end{equation}
where $\{\Psi_n\}_{n \geq 1}$ is a sequence of iid random Lipschitz functions, $\Psi_n \colon \RR \to \RR$ and $R_0$ is a random variable independent of these functions. Note that, if the functions are of the form $\Psi_n(x) = A_n x +B_n$, recursion~\eqref{eq:defIFS} boils down to~\eqref{eq:defXn}. Let $\Psi$ denote a generic element of $\{\Psi_n\}_{n \geq 1}$. As argued by Elton~\cite{elton1990multiplicative}, under some mild moment assumptions on $\Psi$, $\{R_n\}_{n \geq 0}$ has a unique stationary distribution which, realized by random variable $R$, satisfies
\begin{equation}\label{eq:sfpe}
	R  \stackrel{d}{=}   \Psi (R) \quad R \mbox{ independent of }\Psi.
\end{equation}
Our key assumption concerning the function $\Psi$ is being Lipschitz with the Lipschitz constant 
$$
	L={\rm Lip}(\Psi) = \sup_{t\neq s} \frac{|\Psi(t)-\Psi(s)|}{|t-s|}
$$ 
satisfying
\begin{equation}\label{eq:psi0}
	\EE[\log L] <0.
\end{equation}
 Aiming to build upon observations made in previous section, we would like to argue that $R$ posses  
the same tail asymptotic as $X$ if $\Psi$ is close to the function $x \mapsto Ax +B$. This can be achieved in several ways, for example consider $\Psi$ of the form
\begin{equation}\label{eq:psi1}
	\Psi(x) =  Ax +  \Phi(x), \quad  \mbox{for } x \in {\rm supp}\mathcal{L}(R)
\end{equation}
where
 \begin{equation}\label{eq:psi15}
	|\Phi(x)| \leq B \phi(|x|), \quad \PP[B\vee |A|>t]= O(\PP[A>t])
\end{equation}
as $t \to \infty$ and
\begin{equation}\label{eq:psi2}
	\phi \mbox{ is locally bounded} \quad \mbox{ and } \quad \phi(x) = o(x), \quad \mbox{as}\ x\to \8.
\end{equation}
Note that, if ${\rm supp}\mathcal{L}(R)$ unbounded, then necessary $|A| \leq L$. Indeed, writing 
$$
	L\geq \frac{|\Psi (x)-\Psi (0)|}{|x|}=\frac{|Ax+\Phi (x)-\Phi (0)|}{|x|}
$$ 
we notice that
$$
	\lim _{|x|\to \infty }\frac{|\Psi (x)-\Psi (0)|}{|x|}=|A|.
$$
We will assume that $A_+$ satisfies the conditions~\eqref{eq:Atail1}-\eqref{eq:Amoment1}, that is for $ y >0$  
\begin{equation}\label{eq:Atail1*}
	\PP[A> ty] \sim y^{-\alpha}\PP[A>t]
\end{equation}
as $t \to \infty$ and 
\begin{equation}\label{eq:Amoment1*}
	\EE[|A|^{\alpha}] <1.
\end{equation}
Denoting by  $A'$ an independent copy of $A$, we will also suppose that 
\begin{equation}\label{eq:tailA2*}
\PP[A_+A'_+>t] \sim 2 \EE[A_+^{\alpha}]\PP[A>t]
\end{equation}
for some $\alpha >0$. To generalize the condition~\eqref{eq:tailequiv+} and \eqref{eq:tailequiv-} assume for some measurable functions $f_+, f_-\colon \RR \to [0, +\infty)$,
\begin{equation}\label{eq:f+}
	\PP[\Psi(y) >t] \sim f_+(y) \PP[A>t], \quad \mbox{for } y \in {\rm supp}\mathcal{L}(R),
\end{equation} 
and
\begin{equation}\label{eq:f-}
	\PP[\Psi(y) <-t] \sim f_-(y) \PP[A>t], \quad \mbox{for } y \in {\rm supp}\mathcal{L}(R),
\end{equation}
where $\mathcal{L}(R)$ denotes the law of $R$.  
Assuming the above we will prove our main result, which was already foreshadowed by the previous Section.

\begin{thm}\label{thm:ifs}
	Assume that $\Psi$ is a Lipschitz function satisfying~\eqref{eq:psi0}-\eqref{eq:tailA2*}. Then the Markov chain $\{ R_n \}_{n\geq 1}$ converges weakly
	 to $R$, which is a unique solution to~\eqref{eq:sfpe}. Moreover
	\begin{equation*}
		 \PP[|R|>t] = O( \PP[A>t]).
	\end{equation*} 
	Moreover, 
	\begin{itemize}
			\item[a)] Suppose additionally that $\Psi $ is non-decreasing and $\PP[A>t, \: \Phi(1) <-t] = o(\PP[A>t])$ then
					\begin{equation*}
		 				\PP[|R|>t]\asymp \PP[A>t].
					\end{equation*}
			\item[b)] Suppose that \eqref{eq:f+} and \eqref{eq:f-} hold. Then
					\begin{equation}\label{eq:main}
						\PP[R>t] \sim  \frac{ (1-\mu_+)\EE[f_+(R)] + \mu_-\EE[f_-(R)] }{(1 -\mu_+-\mu_-)(1-\mu_+ +\mu_-)  } \PP[A>t]. 
					\end{equation}
	\end{itemize}
\end{thm}

One novelty of our result is that it allows the non-linear term in $\Psi$ to have a substantial contribution.
\begin{ex}
	Suppose that $\Psi$ has the following form
	$$
		\Psi(x) = Ax + B \sqrt{x} \log^+(x) + C,
	$$
	where $A, B, C\geq 0$ are independent and
	$$
		\PP[B>t] \sim c_B\PP[A>t] \quad \mbox{and} \quad \PP[C>t]\sim c_C \PP[A>t].
	$$
	Then, if $A$ satisfies the conditions of Theorem~\ref{thm:ifs},
	$$
		\PP[R>t] \sim \frac{ \EE[R^{\alpha} + c_BR^{\alpha/2} \log^+(R)^{\alpha}] +c_C}{1-\EE[A^{\alpha}]} \PP[A>t].
	$$
\end{ex}  

\begin{ex}
	Consider $\Psi$ of the form
	$$
		\Psi(x) =  \max \{  Ax, B  \}.
	$$
	Then~\eqref{eq:sfpe} can be expressed as 
	$$
		R \stackrel{d}{=} \max \{ A R, B  \} \quad R \mbox{ independent of } (A,B).
	$$
	In this special instance, $R$ shares a distribution with a supremum of a perturbed multiplicative random walk, that is
	$$
		R \stackrel{d}{=} \sup_{k \geq 0} \{ A_1 \ldots A_kB_{k+1} \}.
	$$ 
	Assume for simplicity that $A, B >0$. Then ${\rm supp} \mathcal{L}(R) \subseteq [0, +\infty)$. For $x \geq 0$ we may write
	$$
		0\leq \Psi(x) - Ax = (B- Ax)_+ \leq B
	$$
	and so \eqref{eq:psi1} and \eqref{eq:psi15} are satisfied. If the assumptions of Theorem~\ref{thm:ifs} are satisfied, that is among others
	$$
		\PP[\max \{  Ax, B  \}>t] \sim f_+(x) \PP[A>t],
	$$
	we can infer the asymptotic of the form~\eqref{eq:main} which is a multiplicative equivalent of the result obtained in~\cite{palmowski2007tail}.
\end{ex}
\begin{ex}
	Take $\Psi$ given viz.
	$$
		\Psi(x) = Ax^++B,
	$$
	where $A >0$ and $B > -b$ for some fixed positive $b$. Then $R$ solves
	$$
		R \stackrel{d}{=}  A R^++ B  \quad R \mbox{ independent of } (A,B).
	$$
	Also in this case, the distribution of $R$ has a very particular representation, being the supremum of the perpetuity sequence, that is 
	$$
		R \stackrel{d}{=} \sup_{n \geq 0} \left\{ \sum_{k=0}^nA_1 \ldots A_kB_{k+1} \right\}.
	$$ 
	Random variables of this form have connections to the ruin problem, for details see~\cite{collamore2009random}.
	Since $B > -b$, we know that ${\rm supp} \mathcal{L}(R) \subseteq [-b, +\infty)$. For $x > -b$ one has
	$$
		0\leq \Psi(x) - Ax \leq Ab+B.
	$$
	Again, if the assumptions of Theorem~\ref{thm:ifs} are satisfied, which means that among others
	$$
		\PP[ Ax^++ B >t] \sim f_+(x) \PP[A>t],
	$$
	we can infer the asymptotic similar to~\eqref{eq:main}.
\end{ex}

\section{Proofs}\label{sec:proofs}

In order to establish all of our claims, we will proceed in the following fashion. We will prove the entire Theorem~\ref{thm:ifs}. From this, Proposition~\ref{prop:weak} and Theorem~\ref{thm:perpetuity} will follow. Firstly note, that convergence in~\eqref{eq:Atail1} is uniform in the following sense
\begin{equation}\label{eq:Atail1unif}
	\sup_{y > c} \left(\frac{\PP[A>yt]}{\PP[A>t]} - y^{-\alpha} \right) \to 0
\end{equation}
as $t \to \infty$ for any $c>0$. See for example Bingham et al.~\cite{bingham1989regular}  We begin by noting that the convergence of $\{ R_n\}_{n \geq 0}$ follows form the result of  Elton~\cite{elton1990multiplicative}. More precisely, note that~\eqref{eq:psi0} reads
$$
	\EE[\log(L)]<0
$$
and that~\eqref{eq:psi1} and \eqref{eq:psi15} imply that for some $x_0 \in \RR$,
$$
	\EE[\log |x_0 - \Psi(x_0)|] \leq  \EE[\log(1+|A|)] + \log(1+|x_0|) + \EE[\log(1+|B|)]+\log(1+|\phi(x_0))| <\infty.
$$
The main result of Elton~\cite{elton1990multiplicative} implies the next Proposition.
\begin{prop}\label{prop:1st}
	Assume that $\Psi$ 
	satisfies
	$$
		\EE[\log^+L]<\infty, \quad \EE[\log L] <0 \quad \mbox{and} \quad \EE[\log|x_0 - \Psi(x_0)|] <\infty
	$$ 
	for some $x_0 \in \RR$. Then the Markov chain $\{ R_n \}_{n\geq 1}$ converges weakly, to $R$, which is a unique solution to~\eqref{eq:sfpe}.
\end{prop}

We will now establish a weak tail equivalence of $R$ and $A$.
\begin{prop}\label{prop:2nd}
	Assume that $\Psi$ 
satisfies \eqref{eq:psi0}-\eqref{eq:tailA2*}. Then
	\begin{equation*}
		 \PP[|R|>t] = O( \PP[A>t]).
	\end{equation*} 
	If we suppose additionally $\Psi $ is non-decreasing, $\PP[A>t, \: \Phi(1) <-t] = o(\PP[A>t])$ then
	\begin{equation*}
		 \PP[|R|>t]\asymp \PP[A>t].
	\end{equation*} 
\end{prop}
\begin{proof}
	For the first claim take $\delta >0$ small enough for
	\begin{equation*}
		a_{\delta} = \EE[(|A|+\delta B)^{\alpha}](1-\delta)^{-\alpha} <1.
	\end{equation*}	 
 	Next, pick $t_1$, for which $|\phi(t)| \leq \delta |t| + t_1$. Then it is true that for $\mathcal{L}(R)$ - a.a.  $x \in \RR$,
	$$
		|\Psi(x)| \leq A^*|x|+B^*\quad a.s. 
	$$
	where  $A^* = |A|+\delta B$ and $B^* = Bt_1$. Note that by our assumptions $\PP[A^*>t], \PP[B^*>t] = O(\PP[A>t])$.
	It is true that for any $t\in\RR$, if $R$ is independent from $\{\Psi_n \}_{n \geq 0}$,
	$$
		\PP[|R|>t] = \PP[|\Psi_1(R)|>t] \leq \PP[A_1^*|R|+B_1^*>t],
	$$
	which means that
	$$
		|R| \leq_{st} A_1^*|R|+B^*_1,
	$$
	where $\leq_{st}$ denotes the stochastic order, i.e. $U \leq_{st} V$ iff $\PP[U>t] \leq \PP[V>t]$ for any $t \in \RR$. Since, due to independence of $R$ and $\Psi_2$, it is also true that 
	$|R| \leq_{st} A_2^*|R|+B^*_2$, we can infer by the merit of $A_1^*$ being positive that
	$$
		|R| \leq_{st} A_1^* A_2^*|R| + A_1^*B_2^* + B_1^*.
	$$
	Inductively, we can show this way that for any $n\geq 1$,
	\begin{equation}\label{eq:stbound}
		|R| \leq_{st} A_1^*A_2^* \ldots A_n^* |R|+ \sum_{k=0}^{n-1}A_1^*A_2^* \ldots A_k^*B_{k+1}^*.
	\end{equation}
	Since $A_1^*A_2^* \ldots A_n^* |R|$ converges in probability to $0$ as $n \to \infty$, if we pass to the limit in~\eqref{eq:stbound} we get
	$$
		|R| \leq_{st} X^*=\sum_{k=0}^{\infty}A_1^*A_2^* \ldots A_k^*B_{k+1}^*.
	$$
	From now, we will focus on delivering the bound for the tail of $X^*$. The key observation is that
	$$
		X^* \stackrel{d}{=} A^*X^*+B^*, \quad X^*\mbox{ independent of }(A^*,B^*).
	$$
	which means that $X^*$ is the unique stationary distribution of the Markov chain given via
	$$
		X_{n}^*  = A^*_{n}X^*_{n-1} + B_n^* \quad n \geq 1,
	$$
	where $X_0^*$ is independent of the sequence of iid two-dimensional random vectors $\{(A_n^*, B_n^*)\}_{n \geq 0}$. By Proposition~\ref{prop:1st}, the $X_n^*$ converges weakly to $X^*$ for 
	any choice of $X_0^*$. Form here, we will follow an idea presented previously by Grey~ \cite{grey1994regular}.
	Consider $Y = TA'\ind_{\{A'>y\}}$, where $A'$ is 
	an independent copy of $A$ and $T,y$ are some large constants.	We have
	\begin{equation*}
		\PP[Y>t] \sim T^{\alpha}\PP[A>t] \quad \mbox{and} \quad \EE[Y^{\alpha}]=T^{\alpha}\EE[A^{\alpha}\ind_{\{ A>y\}}].
	\end{equation*}
	By Lemma \ref{lem:convsa}  we can write for some constant $c = c(\delta)$,
	\begin{align*}
		\PP[ A^*Y + B^*>t] &  \leq \PP[A^*Y>(1-\delta)t] + \PP[ B^*>\delta t] \\ 
		& \leq  \left(a_{\delta} T^{\alpha} + c\EE[Y^{\alpha}](1-\delta)^{-\alpha} +c+cT^{\a} o(1)\right) \PP[A>t]\\
		& \leq  \left(a_{\delta} T^{\alpha} + cT^{\alpha}\EE[A^{\alpha}\ind_{\{ A>y\}}](1-\delta)^{-\alpha} + c+cT^{\a}o(1)\right) \PP[A>t].
	\end{align*}
	First, pass with $t \to \infty$ and get
	$$
		\limsup_{t \to \infty}\frac{\PP[ A^*Y + B^*>t]}{\PP[A>t]} \leq a_{\delta} T^{\alpha} + cT^{\alpha}\EE[A^{\alpha}\ind_{\{ A>y\}}](1-\delta)^{-\alpha}.
	$$ 
	For large $T$ and an appropriate choice of $y =y(T)$ we can ensure $c\EE[A^{\alpha}\ind_{\{ A>y\}}] (1-\delta )^{-\a }< 1-a_{\delta }$ and obtain
	\begin{equation*}
		a_{\delta} T^{\alpha} + cT^{\alpha}\EE[A^{\alpha}\ind_{\{ A>y\}}](1-\delta)^{-\alpha} < T^{\alpha}.
	\end{equation*}
	This results in
	$$
		\limsup_{t \to \infty}\frac{\PP[ A^*Y + B^*>t]}{\PP[A>t]} < \lim_{t \to \infty}\frac{\PP[Y>t]}{\PP[A>t]}.
	$$
	Whence we can pick $t_0$, such that for $t>t_0$
	\begin{equation*}
		\PP[Y>t] \geq\PP[A^*Y+B^*>t].
	\end{equation*}
	Define the law of r. v. $X^*_0$ via
	\begin{equation*}
		\PP[X^*_0>t] = \PP[Y> t \: | \: Y>t_0], \quad \mbox{for } t \in \RR.
	\end{equation*}
	Then for any $t \in \RR$
	$$
		\PP[X^*_0>t] \geq \PP[A^*X^*_0 +B^*>t].
	$$
	To see that this is in fact true, consider two possibilities, first of which is $ t >t_0$. Then
	\begin{equation*}
		\PP[ A^*X^*_0+B^* > t] = \PP[A^*Y+B^*>t \: | \: Y> t_0] \leq \frac{\PP[A^*Y+B^*>t]}{\PP[Y>t_0]} \leq \frac{\PP[Y> t]}{\PP[Y>t_0]} = \PP[X^*_0>t].
	\end{equation*}
	For $t <t_0$, $\PP[X^*_0>t]=1$ so that $\PP[A^*X^*_0+B^*>t] \leq \PP[X^*_0>t]$ is trivial. Now, inductively we can write for any $n\geq 1$, since $A^*\geq 0$,
	$$
		\PP[X^*_{n+1}>t] = \PP[ A_{n+1}^*X_n^*+B_{n+1}^*>t] \leq \PP[A^*X^*_0+B^*>t] \leq \PP[X^*_0>t].
	$$
	This completes the proof of the upper bound since 
	$$
		\PP[R>t] \leq \PP[X^*>t] \leq\PP[X_0^*>t] \sim \frac{T^{\alpha}}{\PP [Y>t_0]} \PP[A>t].
	$$
	For the lower bound just note that if $R\geq 1$ then $\Psi (R)\geq \Psi (1)$ and so
	\begin{equation*}
		\PP[R>t]  = \PP[\Psi(R)>t] \geq \PP[\Psi(1)>t , \: R \geq 1]  = \PP[R \geq 1] \PP[\Psi(1)> t] 
	\end{equation*}
	where
	\begin{align*}
		\PP[\Psi(1)> t, \: A>0]& = \PP[A +\Phi(1)>t] \geq \PP[A> 2t] - \PP[A> 2t, \: \Phi(1)< -t] \\ 
			& \geq \PP[A> 2t] - \PP[A> t, \: \Phi(1)< -t] \sim 2^{-\alpha}\PP[A>t].
	\end{align*}
\end{proof}
	After establishing $\PP[|R|>t] = O(\PP[A>t])$ it is relatively easy to get the exact asymptotic, provided that one is equipped with~\eqref{eq:f+} and \eqref{eq:f-}. Note that due to the bound
	$$
		f_{\pm}(x) \leq C_{\pm}(|x|^{\alpha}+1),
	$$
	we know that by the merit of the last Proposition,
	$$
		\EE[f_{\pm}(R)] < \infty. 
	$$
\begin{lem}\label{lem:1iterationsup}
	Suppose \eqref{eq:psi0} - \eqref{eq:f-} are satisfied. Denote  
	$$
		\limsup_{t \to \infty} \frac{\PP[R>t]}{\PP[A>t]} = D_{+}
	$$
	and
	$$
		\limsup_{t \to \infty} \frac{\PP[R<-t]}{\PP[A>t]} = D_{-}.
	$$
	 Then for $(\Psi, A, B, \Phi)$  independent of $R$ we have
	\begin{equation*}
		\limsup_{t \to \infty}\frac{\PP[\Psi(R)>t]}{\PP[A>t]} \leq \EE[|A|^{\alpha}\ind_{\{A>0\}}] D_{+} + \EE[|A|^{\alpha}\ind_{\{A<0\}}] D_{-} + \EE[f_+(R)]
	\end{equation*}
	and
	\begin{equation*}
		\limsup_{t \to \infty}\frac{\PP[\Psi(R)<-t]}{\PP[A>t]} \leq \EE[|A|^{\alpha}\ind_{\{A>0\}}] D_{-} + \EE[|A|^{\alpha}\ind_{\{A<0\}}] D_{+} + \EE[f_-(R)].
	\end{equation*}
\end{lem}
\begin{proof}
	The asymptotic of both probabilities $\PP[\Psi(R)>t]$ and $\PP[\Psi(R)<-t]$ can be treated in the same fashion. Whence, we will consider only the first one. 
	Pick $\delta, \varepsilon \in (0,1)$ and large $\gamma >0$.
	Decompose the probability of interest in the following fashion 
	\begin{align*}
		\PP[\Psi(R)>t]&  =  \PP[\Psi(R)>t, \: |R| \leq \gamma] + \PP[\Psi(R)>t, \: R >  \gamma]+\PP[\Psi(R)>t, \: R < - \gamma] \\ & = I_1+I_2+I_3.
	\end{align*}
	For the first term write
	$$
		I_1=\PP[\Psi(R)>t, \: |R| \leq \gamma ] = \int_{[-\gamma, \gamma ]} \PP[\Psi(y) > t] \PP[R \in \ud y].
	$$
	Take $t_0$ such that $\phi(t) \leq \delta t +t_0$.
	Since~\eqref{eq:Atail1unif} holds with with $c=\gamma^{-1}$, we can find a constant $D= D(\gamma, \delta)$ such that for $\mathcal{L}(R)$- a.a. $|y| \leq \gamma$
	\begin{equation*}
		\PP[|\Psi(y)|>t]  \leq \PP[|A||y|+B(\delta |y|+t_0)>t] \leq D(|y|^{\alpha} +1)\PP[A>t].
	\end{equation*}
	Whence, by the dominated convergence Theorem we are allowed to infer that
	$$
		\lim _{t \to \infty} \frac{I_1}{\PP[A>t]}=\EE[f_+(R)\ind _{\{|R|\leq \gamma \}} ]\leq \EE[f_+(R) ].
	$$
	Since $\Psi (R)=AR+\Phi (R)\leq AR+B\phi (|R|)$, to treat the second term write
	\begin{align*}
		I_2 & \leq \PP[(A+\delta B)R+Bt_0>t,\: R>\gamma ] \\ 
			&  \leq \PP[A^*R>(1-\varepsilon)t,\: R>\gamma] +\PP[B^*>\varepsilon t,\: R>\gamma ]= J_1+J_2,
	\end{align*}
	where $A^*= (A+\delta B)$ and $B^* = Bt_0$. Here, we have for some constant $c$,
	$$
		\limsup_{t \to \infty}\frac{J_2}{\PP[A>t]} \leq  c \varepsilon^{-\alpha}\PP[R>\gamma]
	$$
	For $J_1$, let $R_{\gamma }=R\ind _{\{R>\gamma \}}$ and so
	$$
		J_1=\PP [A^*R_{\gamma}>(1-\eps )t, A^*>0].
	$$ 
	Notice that both $A^*$ and $R_{\gamma }$ satisfy assumptions of Lemma \ref{lem:convsa}. Indeed,
	$$
		\limsup _{t\to \8}\frac{\PP [R_{\gamma }>t]}{\PP [A>t]}=D_+,\quad
 		\limsup _{t\to \8}\frac{\PP [A^*>t]}{\PP [A>t]}=C_+^*<\8 $$
	By an appeal to Lemma~\ref{lem:convsa} we get
	$$
		\limsup_{t \to \infty} \frac{\PP[(A+\delta B)R>(1-\varepsilon)t,\: R>\gamma]}{\PP[A>t]} \leq \big (D_{+} \EE[(A^{*})^{\alpha}\ind _{\{A^*>0\}}] + C^*_+\EE[R_{\gamma }^{\alpha}]\big )  
(1-\varepsilon)^{-\alpha}.
	$$
	Finally, we treat $I_3$ in exactly the case fashion as $I_2$ and arrive at
	$$
		\limsup_{t \to \infty} \frac{\PP[(A+\delta B)R>(1-\varepsilon)t,\: R<-\gamma]}{\PP[A>t]} \leq \big (D_{-} \EE[|A^{*}|^{\alpha}\ind _{\{A^*<0\}}] + C^*_-\EE[|R_{-\gamma }|^{\alpha}]\big )  
(1-\varepsilon)^{-\alpha},
	$$
where $R_{-\gamma}=R\ind _{\{R<-\gamma \}}$ and $C_-^*=\limsup _{t\to \8 }\frac{\PP [A^*<-t]}{\PP [A>t]}$.
	This constitutes
	\begin{align*}
		\limsup_{t \to \infty}\frac{\PP[\Psi(R)>t]}{\PP[A>t]}&  \leq \EE[f_+(R)] + 
\big (D_{+} \EE[(A^{*})^{\alpha}\ind _{A^*>0}] + C^*_+\EE[R_{\gamma }^{\alpha}]\big )  
(1-\varepsilon)^{-\alpha}\\
+& \big (D_{-} \EE[|A^{*}|^{\alpha}\ind _{A^*<0}] + C^*_-\EE[|R_{-\gamma }|^{\alpha}]\big )  
(1-\varepsilon)^{-\alpha},
	\end{align*}
	Take $\gamma \to \infty$ and $\delta, \varepsilon \to 0$ to obtain the claim. 
\end{proof}
Using the same decompositions and Fatou's Lemma instead of the dominated convergence Theorem we also have a Lemma corresponding to the lower limits.
\begin{lem}\label{lem:1iterationinf}
	Suppose \eqref{eq:psi0} - \eqref{eq:f-} are satisfied. Denote  
	$$
		\liminf_{t \to \infty} \frac{\PP[R>t]}{\PP[A>t]} = d_{+}
	$$
	and
	$$
		\liminf_{t \to \infty} \frac{\PP[R<-t]}{\PP[A>t]} = d_{-}.
	$$
	 Then for $(\Psi, A, B, \Phi)$  independent of $R$ we have
	\begin{equation*}
		\liminf_{t \to \infty}\frac{\PP[\Psi(R)>t]}{\PP[A>t]} \geq \EE[|A|^{\alpha}\ind_{\{A>0\}}] d_{+} + \EE[|A|^{\alpha}\ind_{\{A<0\}}] d_{-} + \EE[f_+(R)]
	\end{equation*}
	and
	\begin{equation*}
		\liminf_{t \to \infty}\frac{\PP[\Psi(R)<-t]}{\PP[A>t]} \geq \EE[|A|^{\alpha}\ind_{\{A>0\}}] d_{-} + \EE[|A|^{\alpha}\ind_{\{A<0\}}] d_{+} + \EE[f_-(R)].
	\end{equation*}
\end{lem}
\begin{proof}[Proof of Theorem~\ref{thm:ifs}]
	In view of Propositions~\ref{prop:1st} and \ref{prop:2nd}, only \eqref{eq:main}  needs to be proved. 
	Denote 
	$$
		\mu_{\pm} = \EE[A_{\pm}^{\alpha}], \quad \xi_{\pm}=\EE[f_{\pm}(R)].
	$$ 
	The fact that $R$ satisfies~\eqref{eq:sfpe} combined with Lemmas~\ref{lem:1iterationinf} and \ref{lem:1iterationsup} gives us
	$$
		D_++D_- \leq \frac{\xi_+ + \xi_-}{1-\mu_+- \mu_-}
	$$
	and
	$$
		d_++d_- \geq \frac{\xi_+ + \xi_-}{1-\mu_+- \mu_-}.
	$$
	Since $D_+\geq d_+$ and $D_-\geq d_-$, the two inequalities above imply that $D_+ = d_+$ and $D_-=d_-$. Thus, another appeal to Lemmas~\ref{lem:1iterationinf} and \ref{lem:1iterationsup} yields
	\begin{align*}
		D_+ & = \mu_+D_+ + \mu_-D_- + \xi_+\\
		D_- & = \mu_+D_- + \mu_-D_+ + \xi_-.
	\end{align*}
	Since this system can be solved explicitly, this proves our Theorem.
\end{proof}


\section*{Appendix}
\setcounter{section}{0}
\refstepcounter{section}
\renewcommand{\thesection}{\Alph{section}}

Here, we gathered some facts related to the classes $\mathcal{S}(\alpha)$ and $\mathcal{S}_{\R}(\alpha)$ that we used in the article. Recall, that we will consider distribution $F$ with 
right-unbounded support. Write $F^{*n}$ for 
$n$th-convolution of $F$ and $\overline{F}$ for its tail, that is $\overline{F}(x) = 1-F(x)$.  
Before we prove Lemma \ref{lem:convsa} we need the following auxiliary result.
\begin{lem}\label{F}
	Suppose that $F\in \mathcal{S}_{\R}(\alpha)$. Then for any fixed $v>0$, the limit 
	\begin{equation}
		\lim _{x\to \8} \int _v^{x-v}\frac{\overline F(x-y)}{\overline F  (x)}\ dF(y)
	\end{equation}
	exist. Moreover
	\begin{equation}\label{smallint}
		\lim _{v\to \infty }\lim _{x\to \8} \int _v^{x-v}\frac{\overline F(x-y)}{\overline F (x)}\ dF(y)=0.
	\end{equation}
\end{lem}
\begin{proof}
	We proceed as in the proof of Lemma 2.7 in \cite{embrechts1982convolution}. Let $x>2v$ and let $X,Y$ be two independent random variables with law $F$. Then
	\begin{align*}
		\P (X+Y>x)=&\P (X+Y>x, X\leq v)+\P (X+Y>x, v<X\leq x-v)\\
		&+\P (X>x-v, Y>v )+\P (X+Y>x, Y\leq v) \\
		 =&2\P (X+Y>x, X\leq v)+\P (X+Y>x, v<X\leq x-v)\\
		&+\P (X>x-v, Y>v ).
	\end{align*}
	Hence
	\begin{align*}
		\frac{\overline{F^{*2}}(x)}{\overline F (x)}= &2\int _{-\8 }^v\frac{\overline {F}(x-y)}{\overline F (x)}\ dF(y)\\
		&+\int _v^{x-v}\frac{\overline {F}(x-y)}{\overline F (x)}\ dF(y)+\frac{\overline  {F}(x-v)}{\overline F (x)}\overline F (v).
	\end{align*}
	The third term can be managed quite easily as $x \to \infty$, since by the merit of~\eqref{eq:Lalpha},
	$$
		\lim _{x\to \8}\frac{\overline F(x-v)}{\overline F (x)}\overline F (v)= e^{\a v}\overline F (v).
	$$
	For the same reason, by an appeal to the Lebesgue dominated convergence Theorem, we can identify the limit of the first term as
	$$
		\lim _{x\to \8}\int _{-\8 }^v\frac{\overline {F}(x-y)}{\overline F (x)}\ dF(y)=\int _{-\8}^v e^{\a y}\ dF(y).
	$$
	In view of~\eqref{eq:conveq} we are allowed to conclude that
	$$
		2m_{\alpha }(F)=2\int _{-\8}^v e^{\a y}\ dF(y)\\+ \lim _{x\to \8}\int _v^{x-v}\frac{\overline{F}(x-y)}{\overline F (x)}\ dF(y)+e^{\a v}\overline F (v).
	$$
	This proves our first claim. The second one follows by letting $v\to \8$. 
\end{proof}

\begin{lem}\label{lem:G}
	Suppose that $F\in \mathcal{S}_{\R}(\alpha)$ and that $\overline G _i (y)\leq k_i\overline F (y)$ for $y\geq v >0 $ and $i=1,2$. Then for  $x>2v> 0$ one has
	\begin{equation}\label{2G}
		\int _v^{x-v}\frac{\overline  G_1(x-y)}{\overline F (x)}\ dG_2(y)\leq k_1k_2\Big (\frac{\overline {F}(x-v)}{\overline F (x)}\overline F (v)
		+\int _v^{x-v}\frac{\overline {F}(x-y)}{\overline F (x)}\ dF(y)\Big )
	\end{equation}
\end{lem}
\begin{proof}
	We have
	$$
		\int _v^{x-v}\frac{\bar G_1(x-y)}{\bF (x)}\ dG_2(y)\leq k_1\int _v^{x-v}\frac{\bar F(x-y)}{\bF (x)}\ dG_2(y)
	$$
	which can be bounded further by integrating by parts
	\begin{align*}
		\int _v^{x-v}\frac{\overline F(x-y)}{\overline F (x)}\ dG_2(y)=&-\int _v^{x-v}\frac{\overline F(x-y)}{\overline F (x)}\ d\overline G _2(y) =\frac{\overline F(x-v)}{\overline F (x)}\overline G_2(v)\\
			&-\frac{\overline F(v)}{\overline F (x)}\overline G _2(x-v)+\int _{x-v}^v\frac{\overline G_2(x-y)}{\overline F (x)}\ d\overline F(y)\\
				\leq & k_2\Big (\frac{\overline {F}(x-v)}{\overline F (x)}\overline F (v)+\int _{x-v}^v\frac{\overline  F(x-y)}{\overline F (x)}\ d\overline F(y)\Big )\\
			=&k_2\Big (\frac{\overline {F}(x-v)}{\overline F (x)}\overline F (v)+\int _v^{x-v}\frac{\overline F(x-y)}{\overline F (x)}\ dF(y)\Big )
	\end{align*}
	which competes the proof.
\end{proof}

\begin{proof}[Proof of Lemma \ref{lem:convsa}]
	To prove~\eqref{convol} suppose $x>2v$. As in the proof of Lemma \ref{F} we write
	\begin{align*}
		\frac{\overline{G _1 *G_2}(x)}{\overline F (x)}=&\int _{-\8 }^v\frac{\overline G _1(x-y)}{\overline F (x)}\ dG_2(y)\\
									&+\int _{-\8 }^v\frac{\overline G _2(x-y)}{\overline F (x)}\ dG_1(y)\\
									&+\int _v^{x-v}\frac{\overline  G_1(x-y)}{\overline F (x)}\ dG_2(y) +\frac{\overline {G_2}(x-v)\overline G_1 (v)}{\overline F (x)}.
	\end{align*}
	The first term, by the Lebesgue dominated convergence Theorem, tends to
	\begin{align*}
		\lim _{x\to \8}\int _{-\8 }^v\frac{\overline G _1(x-y)}{\overline F (x)}\ dG_2(y)= & 
			\lim _{x\to \8} \int _{-\8 }^v\frac{\overline G _1(x-y)}{\overline F (x-y)}\frac{\overline F (x-y)}{\overline F (x)}\ dG_2(y)\\
		= &  k_1 \int _{-\8 }^ve^{\a y}\ dG_2(y)
	\end{align*}
	Note that the second term can be treated in exactly the same fashion. The third one,  by Lemmas \ref{F} and \ref{lem:G} in negligible, i.e.
	$$
		\lim _{v\to \8}\lim _{x\to \8}\int _v^{x-v}\frac{\overline G_1(x-y)}{\overline F (x)}\ dG_2(y)=0.
	$$
	Finally, for the last term one has 
 	$$
		\lim _{x\to \8}\frac{\overline {G_2}(x-v)\overline G_1 (v)}{\overline F (x)}= k_2e^{\a v}\overline G_1 (v)
	$$
	Letting $v\to \8 $ we obtain \eqref{convol}. The proof of the fact, that $G_i\in S_{\R}(\a)$ whenever $k_i>0$ is exactly the same as that of Lemma 2.7 in~\cite{embrechts1982convolution}. 
	The fact that the distributions there are supported on $[0,\8 )$ doesn't play any role. In order to argue in favour of~\ref{convol2}, fix $\eps >0$ and take $v>0$ big enough such that 
	$\overline{G}_i(t) \leq (k_i +\varepsilon) \overline{F}(t)$. Let $x_0>2v$ be such that for $x>x_0$
	\begin{equation*}
		\frac{\bar {F}(x-v)}{\bF (x)}\bF (v)+\int _v^{x-v}\frac{\bar {F}(x-y)}{\bF (x)}\ dF(y)<\eps .
	\end{equation*}
	Then in view of \eqref{2G} 
	\begin{align*}
		\frac{\overline{G _1 *G_2}(x)}{\overline F (x)}&=\int _{-\8 }^v\frac{\overline G _1(x-y)}{\overline F (x)}\ dG_2(y)\\
			&+\int _{-\8 }^v\frac{\overline G _2(x-y)}{\bF (x)}\ dG_1(y)+k_1k_2\eps . 
	\end{align*}
	Keeping $v$ fixed and taking $x_0$ possibly larger we have
	\begin{align*}
		\frac{\overline{G _1 *G_2}(x)}{\overline F (x)}&\leq (k_1+\varepsilon)m_{\a}(G_2)+(k_2+\varepsilon)m_{\a }(G_1)+k_1k_2\eps\\
&= k_1m_{\a}(G_2)+k_2m_{\a }(G_1)+(k_1k_2+m_{\a }(G_1)+m_{\a }(G_2))\eps 
	\end{align*}
	which shows \eqref{convol2} and \eqref{convol1}. \eqref{convol3} and \eqref{convol4} are obtained in the same way.
\end{proof}

\bibliographystyle{plain}
\bibliography{edamek_pdysz_IRFandRVT_bib}

\end{document}